\documentclass[12pt,a4paper]{article}
\usepackage[latin1]{inputenc}
\usepackage{amsmath}
\usepackage{amsthm}
\usepackage{amssymb}
\usepackage{fullpage}
\theoremstyle{plain}
\newtheorem{proposition}{Proposition}[section]
\newtheorem{theorem}{Theorem}[section]

\newtheorem{lemma}{Lemma}[section]

\newtheorem{corollary}{Corollary}[section]
\newtheorem{example}{Example}[section]
\numberwithin{equation}{section}
\begin{document}
\begin{center} {\bf Extensions and variations of Andrews-Merca identities}\end{center}
%\medskip
\begin{center}
 Darlison Nyirenda$^{1}$%\footnote{Corresponding Author: Darlison.Nyirenda@wits.ac.za} 
and
Beaullah Mugwangwavari$^{2}$%\footnote{712040@students.wits.ac.za}
 \vspace{0.5cm} \\
$^{1}$  The John Knopfmacher Centre for Applicable Analysis and Number Theory, 	University of the Witwatersrand, P.O. Wits 2050, Johannesburg, South Africa.\\
$^{2}$ School of Mathematics, University of the Witwatersrand, P. O. Wits 2050, Johannesburg, South Africa.\\
e-mails: darlison.nyirenda@wits.ac.za, 712040@students.wits.ac.za\\

\end{center}
\begin{abstract}
Recently, Andrews and Merca have given a new combinatorial interpretation of the total number of even parts in all partitions of $n$ into distinct parts. We generalise this result and consider many more variations of their work. We also highlight some connections with the work of Fu and Tang.
\end{abstract}
\section{Introduction}
\noindent A partition of a positive integer $n$ is a representation $ \lambda = \lambda_{1} + \lambda_{2} + \hdots + \lambda_{r}$ where $\lambda_{1} \geq \lambda_{2} \geq \cdots \geq \lambda_{r}\geq 1$ and $ \sum\limits_{i=1}^{r} \lambda_{i} = n $.  A more compact notation for $\lambda$ is $(\mu_{1}^{m_{1}}, \mu_{2}^{m_{2}}, \ldots)$ where $\mu_1 > \mu_2 > \cdots $ and $m_{i}$ is the multiplicity of  the summand $\mu_i$.  For instance, $14 + 14 + 10 + 10 + 7 + 7 + 7 + 1 + 1 + 1 + 1$ can be written as $(14^{2},10^{2},7^{3},1^{4})$.  The summands in a partition are called \textit{parts}.The union of two partitions $\lambda$ and $\beta$ is simply the multiset union $\lambda \cup \beta$ where $\lambda$ and $\beta$ are treated as multisets. 
\noindent Partition identities involving various classes of partitions have been studied. We recall a popular one: For a positive integer $t$, the number of partitions of $n$ into parts appearing at most $t - 1$ times is equal to the number of partitions of $n$ into parts not divisible by $t$. This identity is due to  J. W. L. Glaisher \cite{glaisher}. Glaisher's bijection established a one-to-one correspondence between the aforesaid classes of partitions. We shall let $\phi_{t}$ denote Glaisher's map, mapping the set of partitions into parts appearing at most $t - 1$ times  to partitions into parts not divisible by $t$.  For more on this map and partition identities, we refer the reader to \cite{andrews}.  The case of $t = 2$ has posed interesting questions and various studies conducted in many directions. Recently, Andrews and Merca \cite{andrewsm} considered the partition functions:\\ 
$a(n)$: the total number of even parts in all partitions of $n$ into distinct parts;  \\
$c_{e}(n)$(resp. $c_{o}(n)$): the number of partitions in which exactly one part is repeated and this part is even (resp. odd);  \\
$b_{o}(n)$: the number of partitions of $n$ into an odd number of parts in which the set of even parts has only one element;\\
$b_{e}(n)$: the number of partitions of $n$ into an even number of parts in which the set of even parts has only one element;\\
$b(n) = b_{o}(n) - b_{e}(n)$ and $c(n) = c_{o}(n) - c_{e}(n)$.\\
They  obtained the following result.
\begin{theorem}\label{main2}
For all $n\geq 1$,  $$a(n) = c(n) = (-1)^{n}b(n).$$
\end{theorem}
Note that, if $b^{\prime}(n)$ denotes the number of partitions of $n$ in which the set of even parts is singleton, then $b^{\prime}(n) = b_{o}(n) + b_{e}(n)$. Hence, one of the implications of the theorem is that 
\begin{equation}\label{modu}
2\mid a(n) - b^{\prime}(n).
\end{equation}
In this paper, we generalize  the identitity $a(n) = c(n)$ of Theorem \ref{main2} and also generalize \eqref{modu}. We then consider some variations. As a consequence, some connections with the work of Fu and Tang in \cite{fu} are highlighted and bijective proofs are provided in such cases.\\ We recall the $q$-series notation:\\
For $a, q\in\mathbb{C}$, we have $ (a;q)_{n} = \prod\limits_{i = 0}^{n - 1}(1 - aq^{i})$ for $n\geq 1$ and  $(a;q)_{0} = 1$, and also for $|q| < 1$,
$$(a;q)_{\infty} = \prod\limits_{i = 0}^{\infty}(1 - aq^{i}).$$
Below are some useful $q$-identities:
\begin{equation}\label{pent}
(q;q)_{\infty} = \sum_{j = -\infty}^{\infty}(-1)^{j}q^{\frac{j(3j - 1)}{2}}= 1 + \sum_{j = 1}^{\infty}(-1)^{j}q^{\frac{j(3j\pm 1)}{2}},
\end{equation}
\begin{equation}\label{three}
(q;q)_{\infty}^{3} = \sum_{j = 0}^{\infty}(-1)^{j}(2j + 1)q^{\frac{j(j + 1)}{2}}.
\end{equation}
\section{Generalizing the identity $a(n) = c(n)$ of Theorem \ref{main2} }
We first observe the following.
\begin{theorem}\label{main3}
Suppose that $p$ and $r$ are non-negative integers such that $p \geq r + 2$.  Denote by $a_{r}(n,p)$ the total number of parts congruent to $-r \pmod{p}$ in partitions of $n$ into distinct parts. Let $g_{r}(n,p)$ denote the number of partitions in which exactly one part is repeated and the multiplicity of this repeated part is at least $p - r$ and congruent to $-r, -r + 1 \pmod{p}$. Let $g_{r,o}(n,p)$ (resp. $g_{r,e}(n,p)$) denote the number of $g_{r}(n,p)$-partitions in which the repeated part is odd (resp. even) and let $g_{r}(n,p) = g_{r,o}(n,p) - g_{r,e}(n,p)$. Then  
$$a_{r}(n,p) = g_{r}(n,p).$$
\end{theorem}
\noindent It is important to realize that setting $p = 2$ and  $r = 0$ yields the first equality in Andrews and Merca's theorem, Theorem \ref{main2}.
\begin{proof}
\noindent Let $F(z,q)$ be the multivariable generating function for $g_{r}(n,p)$ in which $z$ tracks the repeated part. Then
\begin{align*}
F(z,q)  & =\sum\limits_{n = 1}^{\infty}z^{n}(q^{(p-r)n} + q^{(p-r + p)n} + q^{(p-r + 2p)n} + \cdots + q^{(p-r + 1)n} +  q^{(p - r + 1 +p)n} \\\\
          & \qquad \qquad \qquad + q^{(p-r + 1 + 2p)n} + \cdots )\prod\limits_{j \neq n, j = 1}^{\infty}(1 + q^{j}) \end{align*}
\begin{align*}
          & =\sum\limits_{n = 1}^{\infty}z^{n}(q^{(p-r)n}( 1  + q^{ pn} + q^{2pn} +  q^{3pn} + \cdots )+ q^{(p-r + 1)} (1 + q^{pn} \\
          & \qquad \qquad \qquad + q^{ 2pn} + q^{ 3pn} + \cdots ))\prod\limits_{j \neq n, j = 1}^{\infty}(1 + q^{j}) \\\\
          & = \sum\limits_{n = 1}^{\infty}z^{n}q^{(p-r)n}(1 + q^{n})( 1  + q^{ pn} + q^{2pn} +  q^{3pn} + \cdots )\prod\limits_{j \neq n, j = 1}^{\infty}(1 + q^{j})\\\\
          & = \sum\limits_{n = 1}^{\infty}\frac{z^{n}q^{(p-r)n}(1 + q^{n})}{1 - q^{pn}} \prod\limits_{j \neq n, j = 1}^{\infty}(1 + q^{j})\\\\
        & = \sum\limits_{n = 1}^{\infty}\frac{z^{n}q^{(p-r)n}}{1 - q^{pn}} \prod\limits_{j = 1}^{\infty}(1 + q^{j})\\\\
        & = (-q;q)_{\infty}\sum\limits_{n = 1}^{\infty}\frac{z^{n}q^{(p-r)n}}{1 - q^{pn}}.
\end{align*} 
Using the fact that $\sum\limits_{n = 0}^{\infty}(g_{r,o}(n,p) - g_{r,e}(n,p))q^{n} = -F(-1,q),$ we have
\begin{align*}
\sum\limits_{n = 0}^{\infty}g_{r}(n,p)q^{n} & = -F(-1,q) \\
                                          & = -(-q;q)_{\infty}\sum\limits_{n = 1}^{\infty}\frac{(-1)^{n}q^{(p-r)n}}{1 - q^{pn}} \\
                                          & = -(-q;q)_{\infty}\sum\limits_{n = 1}^{\infty}(-1)^{n}q^{-rn}\frac{q^{pn}}{1 - q^{pn}} \\
                                          & = -(-q;q)_{\infty}\sum\limits_{n = 1}^{\infty}(-1)^{n}q^{-rn} \sum\limits_{m = 1}^{\infty} q^{pnm} \end{align*}
\begin{align*}
                                          & = -(-q;q)_{\infty} \sum\limits_{m = 1}^{\infty} \sum\limits_{n = 1}^{\infty}(-1)^{n}q^{-rn}q^{pnm} \\
                                          & = -(-q;q)_{\infty} \sum\limits_{m = 1}^{\infty} \sum\limits_{n = 1}^{\infty}(-q^{-r + pm})^{n}\\
                                          & = -(-q;q)_{\infty} \sum\limits_{m = 1}^{\infty} \frac{-q^{-r + pm}}{1 + q^{pm - r}}\\
                                          & = (-q;q)_{\infty} \sum\limits_{m = 1}^{\infty} \frac{q^{-r + pm}}{1 + q^{pm - r}}.
\end{align*}
On the other hand, 
\begin{align*} 
\sum\limits_{n = 0}^{\infty}a_{r}(n,p)q^{n} & = \frac{\partial}{ \partial z}\Big\vert_{z = 1} \prod\limits_{n = 1}^{\infty}(1 + zq^{pn - r})\prod\limits_{i = 1, i \neq r}^{p}\prod\limits_{n = 1}^{\infty}(1 + q^{pn - i}) \\\\
                                          & = \prod\limits_{i = 1, i \neq r}^{p}\prod\limits_{n = 1}^{\infty}(1 + q^{pn - i})\prod\limits_{n = 1}^{\infty}(1 + zq^{pn - r})\Big\vert_{z = 1}\frac{\partial}{ \partial z}\Big\vert_{z = 1}\sum\limits_{n = 1}^{\infty}\log(1 + zq^{pn - r}) \\\\
                                          & = \prod\limits_{n = 1}^{\infty} (1 + q^{n})\sum\limits_{n = 1}^{\infty}\frac{q^{pn - r}}{1 + q^{pn - r}} \\
                                          & = (-q;q)_{\infty} \sum\limits_{n = 1}^{\infty} \frac{q^{-r + pn}}{1 + q^{pn - r}}.
\end{align*}
\noindent Thus $$ \sum\limits_{n = 0}^{\infty}a_{r}(n,p)q^{n} = \sum\limits_{n = 0}^{\infty}g_{r}(n,p)q^{n}$$ and the theorem follows. 
\end{proof}
\noindent The above theorem provides an a new combinatorial interpretation of the total number of parts congruent to $-r$ modulo $p$  in partitions  of $n$ into distinct parts. \\\\
\noindent Recall the function $O_{p,j}(n)$ in \cite{fu}. For $j = 1$, we shall denote this function by $o_{p}(n)$ which counts partitions of $n$ in which the set of parts congruent to 0 modulo $p$ is singleton. Define  $a(n,p)$ to be the total number of parts congruent to 0 modulo $p$ in all partitions of $n$ into parts appearing at most $p - 1$ times. Clearly $a(n,2) = a(n) = a_{0}(n,2)$  and $o_{2}(n) = b^{\prime}(n)$. We have the following property which extends \eqref{modu}.
\begin{theorem}\label{ee}
For all $n\geq 0$,
$$p\mid a(n,p) - o_{p}(n).$$
\end{theorem}
\begin{proof}
The generating function for $o_{p}(n)$ is given by
$$ \sum_{n = 0}^{\infty}o_{p}(n)q^{n} = \frac{(q^p;q^p)_{\infty}}{(q;q)_{\infty}}\sum\limits_{n = 1}^{\infty}\frac{q^{pn}}{1 - q^{pn}}$$
and the generating function for $a(n,p)$ is given by Herden, Sepanski, et al.  \cite{herden}
 $$ \frac{(q^{p};q^{p})_{\infty}}{(q;q)_{\infty}}\sum\limits_{n = 1}^{\infty}\frac{ q^{pn} + 2q^{2pn} + 3q^{3pn} + \ldots (p - 1)q^{(p - 1)pn} } {1 + q^{pn} + q^{2pn} + \ldots + q^{(p - 1)pn}}.$$
Recall the identity:
\begin{equation}\label{form}
\sum_{k = 1}^{n}kx^{k} = \frac{x(1 - x^{n + 1}) - (n + 1)(1 - x)x^{n + 1}}{(1 - x)^2}.
\end{equation}
Manipulating  the generating function, we have
\begin{align*}
\sum_{n = 0}^{\infty}a(n,p)q^{n} &  = \frac{(q^{p};q^{p})_{\infty}}{(q;q)_{\infty}}\sum\limits_{n = 1}^{\infty}\frac{  \sum\limits_{j = 1}^{p - 1}jq^{pnj}}{\frac{1 - q^{p^{2}n}}{1 - q^{pn}} }  \\
                                                   & =  \frac{(q^{p};q^{p})_{\infty}}{(q;q)_{\infty}}\sum\limits_{n = 1}^{\infty}\frac{  \frac{q^{pn}(1 - q^{p^{2}n}) - p(1 - q^{pn})q^{p^{2}n}}{(1 - q^{pn})^2}}{\frac{1 - q^{p^{2}n}}{1 - q^{pn}} }\,\,\,\,(\text{by \eqref{form} with}\,\,\, n = p - 1 , x = q^{pn}) \\
& = \frac{(q^{p};q^{p})_{\infty}}{(q;q)_{\infty}}\sum\limits_{n = 1}^{\infty} \frac{1 - q^{pn}}{1 - q^{p^{2}n}}\,\,\frac{q^{pn}(1 - q^{p^{2}n}) - p(1 - q^{pn})q^{p^{2}n}}{(1 - q^{pn})^2} \\
& \equiv \frac{(q^{p};q^{p})_{\infty}}{(q;q)_{\infty}}\sum\limits_{n = 1}^{\infty} \frac{1 - q^{pn}}{1 - q^{p^{2}n}}\,\,\frac{q^{pn}(1 - q^{p^{2}n}) }{(1 - q^{pn})^2} \pmod{p} \\
&  = \frac{(q^{p};q^{p})_{\infty}}{(q;q)_{\infty}}\sum\limits_{n = 1}^{\infty}\frac{q^{pn}}{1 - q^{pn}} \\
& = \sum\limits_{n = 0}^{\infty}o_{p}(n)q^{n}.
\end{align*}
\end{proof}
Denoting by $s(n)$ the number of (unrestricted)partitions of $n$ and letting  $f(q) = \sum\limits_{n = 1}^{\infty} \frac{q^{n}}{1 - q^{n}}  - p \sum\limits_{n = 1}^{\infty}\frac{q^{pn}}{1 - q^{pn}}$. Then
\begin{align*}
 \sum_{n = 0}^{\infty}a(n,p)q^{n}  & =    (q^{p};q^{p})_{\infty}\left( \sum\limits_{n = 1}^{\infty} \frac{q^{pn}}{1 - q^{pn}}  - p \sum\limits_{n = 1}^{\infty}\frac{q^{p^{2}n}}{1 - q^{p^{2}n}}\right)\frac{1}{(q;q)_{\infty}}    \\
                                                     &  = (q^{p};q^{p})_{\infty} f(q^{p})\sum_{n = 0}^{\infty}s(n)q^{n}  \\
                                                   & = \sum_{j = 0}^{p - 1}q^{j}(q^{p};q^{p})_{\infty} f(q^{p})\sum_{n = 0}^{\infty}s(pn + j)q^{pn}
\end{align*}
so that
$$\sum\limits_{j = 0}^{p - 1}q^{j}\sum_{n = 0}^{\infty}a(pn + j,p)q^{pn}  = \sum_{j = 0}^{p - 1}q^{j}(q^{p};q^{p})_{\infty} f(q^{p})\sum_{n = 0}^{\infty}s(pn + j)q^{pn} $$
which implies that, for a fixed  $0\leq j\leq p-1$,
$$\sum_{n = 0}^{\infty}a(pn + j,p)q^{pn +j}  = (q^{p};q^{p})_{\infty} f(q^{p})\sum_{n = 0}^{\infty}s(pn + j)q^{pn + j}$$
so that
$$ \sum_{n = 0}^{\infty}a(pn + j,p)q^{pn}  = (q^{p};q^{p})_{\infty} f(q^{p})\sum_{n = 0}^{\infty}s(pn + j)q^{pn}.$$
Replacing $q$ with $q^{1/p}$ yields
$$ \sum_{n = 0}^{\infty}a(pn + j,p)q^{n}  = (q;q)_{\infty} f(q)\sum_{n = 0}^{\infty}s(pn + j)q^{n}.$$
Thus, for  $n\geq 0$, 
$$s(pn +j) \equiv 0 \pmod{p}\,\,\Rightarrow\,\, a(pn +j, p) \equiv 0 \pmod{p}.$$
 In particular, applying the Ramanujan's partition congruences: $s(5n + 4) \equiv 0 \pmod{5}$, $s(7n + 5) \equiv 0 \pmod{7}$ and  $s(11n + 6) \equiv 0 \pmod{11}$, we have the following result.
\begin{theorem}
For all $n\geq 0$,
$$a(5n + 4, 5) \equiv 0 \pmod{5},$$
$$ a(7n + 5, 7) \equiv 0 \pmod{7}$$
and
$$ a(11n + 6,11) \equiv 0 \pmod{11}.$$
\end{theorem}

\vspace{5mm}

\noindent Let $o_{p,e}(n)$ (resp. $o_{p,o}(n)$) be the number of $o_{p}(n)$-partitions in which the number of parts congruent to 0 modulo $p$  is even (resp. odd) and 
\vspace{4mm}

\begin{equation*}
H(z,q) = \frac{(q^{p};q^{p})_{\infty}}{(q;q)_{\infty}}\sum_{n = 1}^{\infty} \frac{zq^{pn }}{1 - zq^{pn }}.
\end{equation*}
Then
\begin{align}
\sum_{n = 0}^{\infty}o_{p,o}(n)q^{n} & = \frac{1}{2}\left(H(1,q) - H(-1,q)\right) \nonumber\\
                                                           & = \frac{(q^{p};q^{p})_{\infty}}{2(q;q)_{\infty}}\left(\sum_{n = 1}^{\infty} \frac{q^{pn }}{1 - q^{pn }}  + \sum_{n = 1}^{\infty} \frac{q^{pn }}{1 + q^{pn }}  \right)\nonumber \\
                                                           & = \frac{(q^{p};q^{p})_{\infty}}{(q;q)_{\infty}}\sum_{n = 1}^{\infty} \frac{q^{pn }}{1 - q^{2pn }}\nonumber
\end{align}
\begin{align}\label{later1}
                                                           & = \frac{(q^{p};q^{p})_{\infty}}{(q;q)_{\infty}}\sum_{n = 1}^{\infty} q^{pn } \sum_{m = 0}^{\infty}q^{2pnm}\nonumber\\
                                                          & =  \frac{(q^{p};q^{p})_{\infty}}{(q;q)_{\infty}}\sum_{m = 0}^{\infty}\sum_{n = 1}^{\infty}q^{(p + 2pm)n}\nonumber\\
                                                          & = \frac{(q^{p};q^{p})_{\infty}}{(q;q)_{\infty}}\sum_{m = 0}^{\infty}\frac{q^{p + 2pm}}{1 - q^{2pm + p}}
\end{align}
and similarly,
\begin{align}\label{later2}
\sum_{n = 0}^{\infty}o_{p,e}(n)q^{n} & = \frac{1}{2}\left(H(1,q) +  H(-1,q)\right) \nonumber\\
                                                               & = \frac{(q^{p};q^{p})_{\infty}}{(q;q)_{\infty}}\sum_{n = 1}^{\infty} \frac{q^{2pn }}{1 - q^{2pn }}.
\end{align}
Thus, if for $ i = 0, p$, $h_{i}(n,p)$ is the number of partitions in which parts are congruent  to $1,2,3,\ldots, p - 1$ modulo $p$  or congruent to $i$ modulo $2p$ and the set of parts $\equiv i \pmod{2p}$ is singleton, then
\begin{equation*}
o_{p,o}(n)= h_{p}(n,p) \,\,\,\,\,\text{and}\,\,\,\,o_{p,e}(n) = h_{0}(n,p).
\end{equation*}
\section{Further variations}
Consider the following partition functions:\\
$d_{e}(n)$: the number of partitions of $n$ in which exactly one even part is repeated and odd parts are unrestricted. \\
$d_{o}(n)$: the number of partitions of $n$ in which exactly one odd part is repeated and even parts are unrestricted. \\
For $i = 0, 2$, $f_{i}(n)$: the number of partitions of $n$ in which the set of parts congruent to $i \pmod{4}$ is singleton. 
Then we have the following theorem:
\begin{theorem}\label{var0}
For $n \geq 1$,  $$d_{e}(n) = f_{0}(n) \,\,\,\,\,\text{and}\,\,\,\,\,\,d_{o}(n) = f_{2}(n).$$
\end{theorem}
\begin{proof}
Note that
\begin{align*}
\sum\limits_{n = 0}^{\infty}d_{e}(n)q^{n} & = \frac{1}{(q;q^{2})_{\infty}}\sum\limits_{n = 1}^{\infty} \frac{q^{2n + 2n}}{1 - q^{2n}}\prod\limits_{j \neq n, j = 1}^{\infty}(1 + q^{2j}) \\
                                          & = \frac{1}{(q;q^{2})_{\infty}}\sum\limits_{n = 1}^{\infty} \frac{q^{4n}(-q^{2};q^{2})_{\infty}}{(1 - q^{2n})(1 + q^{2n})}\\
                                          & = \frac{(-q^{2};q^{2})_{\infty}}{(q;q^{2})_{\infty}}\sum\limits_{n = 1}^{\infty} \frac{q^{4n}}{1 - q^{4n}}  \end{align*}
\begin{align*}
                                          & = \frac{(-q^{2};q^{2})_{\infty}(q^{2};q^{2})_{\infty}}{(q;q^{2})_{\infty}(q^{2};q^{2})_{\infty}}\sum\limits_{n = 1}^{\infty} \frac{q^{4n}}{1 - q^{4n}} \\
                                          & = \frac{(q^{4};q^{4})_{\infty}}{(q;q)_{\infty}}\sum\limits_{n = 1}^{\infty} \frac{q^{4n}}{1 - q^{4n}} \\
                                          & = \frac{1}{(q;q^{4})_{\infty} (q^{2};q^{4})_{\infty} (q^{3};q^{4})_{\infty}}\sum\limits_{n = 1}^{\infty} (q^{4n} + q^{4n + 4n} + q^{4n + 4n + 4n} + \ldots ) \\
                                          & = \sum\limits_{n = 0}^{\infty}f_{0}(n)q^{n}.
\end{align*}
Similarly,
\begin{align*}
\sum\limits_{n = 0}^{\infty}d_{o}(n)q^{n} & = \frac{1}{(q^{2};q^{2})_{\infty}}\sum\limits_{n = 1}^{\infty} \frac{q^{2n - 1 + 2n - 1}}{1 - q^{2n - 1}}\prod\limits_{j \neq n, j = 1}^{\infty}(1 + q^{2j - 1}) \\\\
                                          & = \frac{1}{(q^{2};q^{2})_{\infty}}\sum\limits_{n = 1}^{\infty} \frac{q^{4n - 2}(-q;q^{2})_{\infty}}{(1 - q^{2n - 1})(1 + q^{2n - 1})} \\\\
                                          & = \frac{(-q ;q^{2})_{\infty}}{(q^{2};q^{2})_{\infty}}\sum\limits_{n = 1}^{\infty} \frac{q^{4n - 2}}{1 - q^{4n - 2}} \\
                                          & = \frac{(-q;q^{2})_{\infty}(q;q^{2})_{\infty}}{(q^{2};q^{2})_{\infty}(q;q^{2})_{\infty}}\sum\limits_{n = 1}^{\infty} \frac{q^{4n - 2}}{1 - q^{4n - 2}} \\
                                           & = \frac{(q^{2};q^{4})_{\infty}}{(q;q)_{\infty}}\sum\limits_{n = 1}^{\infty} \frac{q^{4n - 2}}{1 - q^{4n - 2}} \\
                                          & = \frac{1}{(q;q^{4})_{\infty} (q^{3};q^{4})_{\infty} (q^{4};q^{4})_{\infty}}\sum\limits_{n = 1}^{\infty} (q^{4n - 2} + q^{2(4n - 2)} + q^{3(4n - 2)} + \ldots ) \\
                                          & = \sum\limits_{n = 0}^{\infty}f_{2}(n)q^{n}.
\end{align*}
\end{proof}
\noindent We also note the following bijective proof which is uniform for both identities in the theorem. \\\\
\noindent \underline{The bijection } \\\\
\noindent Let $\lambda = \left(\lambda_{1}^{m_{1}}, \lambda_{2}^{m_{2}}, \hdots, \lambda_{l}^{m_{l}}\right)$ be a partition enumerated by $f_{2}(n)/f_{0}(n)$. Then 
$$
\lambda_{i}^{m_{i}} \mapsto \begin{cases}
\left(\dfrac{\lambda_{i}}{2}\right)^{2m_{i}}, \quad \lambda_{i} \equiv 0,2 \pmod{4}; \\\\
\phi_{2}^{-1}(\lambda_{i}^{m_{i}}), \quad \lambda_{i} \equiv 1,3 \pmod{4}.
\end{cases}
$$

%\begin{example}
%For $n = 7$, set $\{(6, 1), (5,2), (4, 2, 1), (3, 2^{2}), (3, 2, 1^{2}), (2^{3}, 1), (2^{2}, 1^{3}), (2, 1^{5})\}$ is enumerated by $f_{2}(7)$. In partition $(2, 1^{5})$ the parts are congruent to $2 \pmod 4$ and $1 \pmod{4}$, respectively. Applying the map we obtain $$ 2 \mapsto 1^{2},$$ $$ 1^{5} \mapsto (4,1). $$ Taking the union gives $$ (2, 1^{5}) \mapsto (4, 1^{3}). $$ Similarly,
%\[ (6, 1) \mapsto (3^{2}, 1) \]
%\[ (5,2)  \mapsto (5, 1^{2}) \]
%\[ (4, 2, 1) \mapsto (2^{2}, 1^{3}) \]
%\[ (3, 2^{2}) \mapsto (3, 1^{4}) \]
%\[ (3, 2, 1^{2}) \mapsto (3, 2, 1^{2}) \]
%\[ (2^{3}, 1) \mapsto 1^{7} \]
%\[ (2^{2}, 1^{3}) \mapsto (2, 1^{5}) \]
%
%\end{example}

\noindent To invert the mapping, let $\mu = \left(\mu_{1}^{s_{1}}, \mu_{2}^{s_{2}}, \hdots, \mu_{l}^{s_{l}} \right)$ be a partition enumerated by $d_{o}(n)/d_{e}(n)$. Then 
\[
\mu_{i}^{s_{i}} \mapsto  (2\mu_{i})^{\lfloor \frac{s_i}{2}\rfloor } \cup \phi_{2}(\mu_{i}^{ s_i - 2\lfloor \frac{s_i}{2}\rfloor }).\]
\vspace{5mm}

\noindent Recall that $$ \sum\limits_{n = 0}^{\infty}d_{e}(n)q^{n} = \frac{(q^{4};q^{4})_{\infty}}{(q;q)_{\infty}}\sum\limits_{n = 1}^{\infty} \frac{q^{4n}}{1 - q^{4n}}$$ so that
\begin{equation}\label{eql}
(q;q)_{\infty}\sum\limits_{n = 0}^{\infty}d_{e}(n)q^{n} = (q^{4};q^{4})_{\infty}\sum\limits_{n = 1}^{\infty} \frac{q^{4n}}{1 - q^{4n}}.
\end{equation}
Using \eqref{pent}, we have 
\begin{align}\label{ine}
(q;q)_{\infty}\sum\limits_{n = 0}^{\infty}d_{e}(n)q^{n}  & =  \left( 1 + \sum\limits_{n = 1}^{\infty}(-1)^{n}q^{\frac{n(3n \pm 1)}{2}} \right)\sum\limits_{n = 0}^{\infty}d_{e}(n)q^{n} \nonumber \\
& = \sum_{n = 0}^{\infty}\sum_{k = 0}^{n}c_{k}d_{e}(n-k)q^{n}
\end{align}
where  
\[ c_{k} = 
\begin{cases}
 (-1)^{j}, & \text{ if $k =\frac{ j(3j \pm 1)}{2}, j \in \mathbb{Z}_{\geq 0}$}; \\
 0, & \text{otherwise}.
\end{cases}
\]
Note that the right-hand side of \eqref{ine} is equal to   
 $$\sum\limits_{n = 0}^{\infty} \left[d_{e}(n) + \sum\limits_{j = 1}^{\lfloor\frac{1 + \sqrt{24n + 1}}{6} \rfloor} (-1)^{j}\left( d_{e}\left(n - \frac{j(3j + 1)}{2}\right) + d_{e}\left(n - \frac{j(3j - 1 )}{2}\right)\right)\right]q^{n}.$$
\noindent Invoking \eqref{pent} and the identity
$$ \sum\limits_{n = 1}^{\infty}\frac{q^{n}}{1 - q^{n}}  =  \sum\limits_{n = 1}^{\infty}\sigma_{0}(n)q^{n}$$ where $\sigma_{0}(n)$ is the number of  positive divisors of $n$,  \eqref{eql} becomes
\begin{align*}
 	& \sum\limits_{n = -\infty}^{\infty} (-1)^{n}q^{n(3n + 1)/2} \cdot \sum\limits_{n = 0}^{\infty}d_{e}(n)q^{n}\\
 & = \sum\limits_{n = -\infty}^{\infty}(-1)^{n}q^{2n(3n + 1)} \sum\limits_{n = 1}^{\infty} \sigma_{0}(n)q^{4n}.
                                                    \end{align*}
Hence, we have the following.
\begin{corollary}
For all integers  $r \geq 0$, let  $$ A(r) = \{j \in \mathbb{Z}_{> 0}:  2j(3j + 1) \equiv r \pmod{4} \,\,\,\text{and}\,\,\, j \leq  \left\lfloor\frac{\sqrt{6r + 1} - 1}{6}\right\rfloor \}$$ and 
$$ B(r) = \{j \in \mathbb{Z}_{> 0}:  2j(3j - 1) \equiv r \pmod{4} \,\,\,\text{and}\,\,\, j\leq \left\lfloor\frac{ 1 + \sqrt{6r + 1}}{6} \right\rfloor \}.$$
Then for a positive integer $n$, 
\[
d_{e}(n) 
= \begin{cases} 
\sum\limits_{j =1}^{ \lfloor\frac{ 1 + \sqrt{24n + 1}}{6} \rfloor  } (-1)^{j + 1 }\left(d_{e}\left(n - \frac{j(3j + 1 )}{2}\right) + d_{e}\left(n - \frac{j(3j - 1 )}{2}\right)\right)  & \text{if $n \not\equiv 0 \pmod{4}$} \\\\
\sum\limits_{j =1}^{ \lfloor\frac{ 1 + \sqrt{24n + 1}}{6} \rfloor  } (-1)^{j + 1 }\left(d_{e}\left(n - \frac{j(3j + 1 )}{2}\right) + d_{e}\left(n - \frac{j(3j - 1 )}{2}\right)\right)  + \gamma(n) & \text{if $ n \equiv 0 \pmod{4}$}
\end{cases}
\]
where $$ \gamma(n) = \sigma_{0}(n/4)  + \sum_{j\in A(n)}(-1)^{j}\sigma_{0}\left(\frac{n - 2j(3j + 1)}{4}\right)  +  \sum_{j\in B(n)}(-1)^{j} \sigma_{0}\left(\frac{n - 2j(3j - 1)}{4}\right).$$
\end{corollary}
\noindent In this case, we have a relationship between $d_{e}(n)$ and $\sigma_{0}(n)$.
\begin{example}
Consider $n=8$.
\end{example}
%\noindent The $d_{e}(8)$-partitions are:
%$$...fill in the partitions here.....$$
\noindent Since $8 \equiv 0 \pmod{4}$, we have
\begin{align*}
d_{e}(8) & = d_{e}(8-2) + d_{e}(8-1) - d_{e}(8-7) - d_{e}(8-5) + \gamma(8) \\
 & = d_{e}(6) + d_{e}(7) - d_{e}(1) - d_{e}(3) + \gamma(8) \\
% & = [d_{e}(6-2) + d_{e}(6-1) - d_{e}(6-7)- d_{e}(6-5)] + [d_{e}(7-2) + d_{e}(7-1) - d_{e}(7-7) - d_{e}(7-5)] - [d_{e}(1-2) + d_{e}(1-1)] - [d_{e}(3-2) + d_{e}(3-1)] + \gamma(8) \\
 & = d_{e}(4) + 2d_{e}(5) - d_{e}(1) + d_{e}(6) - 2d_{e}(2) - d_{e}(1) + \gamma(8) \\
% & = [d_{e}(4-2) + d_{e}(4-1) + \gamma(4)] + 2[d_{e}(5-2) + d_{e}(5-1) - d_{e}(5-7) - d_{e}(5-5)] - [d_{e}(1-2) + d_{e}(1-1)] + [d_{e}(6-2) + d_{e}(6-1) - d_{e}(6-7)- d_{e}(6-5)] - 2[d_{e}(2-2) + d_{e}(2-1)] - [d_{e}(1-2) + d_{e}(1-1)] + \gamma(8) \\
 & = d_{e}(2) + 3d_{e}(3) + 3d_{e}(4) + d_{e}(5) - 3d_{e}(1) + \gamma(4) + \gamma(8) \\
% & = [d_{e}(2-2) + d_{e}(2-1)] + 3[d_{e}(3-2) + d_{e}(3-1)] + 3[d_{e}(4-2) + d_{e}(4-1) + \gamma(4)] + [d_{e}(5-2) + d_{e}(5-1) - d_{e}(5-7) - d_{e}(5-5)] - 3[d_{e}(1-2) + d_{e}(1-1)] + \gamma(4) + \gamma(8) \\
 & = 4d_{e}(1) + 6d_{e}(2) + 4d_{e}(3) + d_{e}(4) + 4\gamma(4) + \gamma(8) \\ 
% & = 4[d_{e}(1-2) + d_{e}(1-1)] + 6[d_{e}(2-2) + d_{e}(2-1)] + 4[d_{e}(3-2) + d_{e}(3-1)] + [d_{e}(4-2) + d_{e}(4-1) + \gamma(4)] + 4\gamma(4) + \gamma(8) \\ 
 & = 10d_{e}(1) + 5d_{e}(2) + d_{e}(3) + 5\gamma(4) + \gamma(8) \\ 
% & = 10[d_{e}(1-2) + d_{e}(1-1)] + 5[d_{e}(2-2) + d_{e}(2-1)] + [d_{e}(3-2) + d_{e}(3-1)] + 5\gamma(4) + \gamma(8) \\ 
 & = 6d_{e}(1) + d_{e}(2) + 5\gamma(4) + \gamma(8) \\ 
% & = 6[d_{e}(1-2) + d_{e}(1-1)] + [d_{e}(2-2) + d_{e}(2-1)] + 5\gamma(4) + \gamma(8) \\
 & = d_{e}(1) + 5\gamma(4) + \gamma(8) \\ 
% & = [d_{e}(1-2) + d_{e}(1-1)] + 5\gamma(4) + \gamma(8) \\
 & = 5\gamma(4) + \gamma(8). 
\end{align*}
\noindent Now, $A(8)=\{1\} , B(8)=\{1\} , A(4)=\{\} $ and $B(4)=\{1\}$  and so 
\begin{align*} 
d_{e}(8) & = 5\gamma(4) + \gamma(8) \\
 & = 5[\sigma_{0}(1) - \sigma_{0}(0)] + [\sigma_{0}(2) + \sigma_{0}(0) - \sigma_{0}(1) \\
 & = \sigma_{0}(2) + 4\sigma_{0}(1) - 5\sigma_{0}(0) \\
 & = 2 + 4(1) - 5(0) \\
 & =  6.
\end{align*} 
%%d_{e}()
%From the table of values for $d_{e}(n)$ below, we see that this is true. \\\\
%\begin{tabular}{|c|c|c|c|c|c|c|c|c|}
%\hline
%$n$ & $1$ & $2$ & $3$ & $4$ & $5$ & $6$ & $7$ & $8$  \\
%\hline
%$d_{e}(n)$ & $0$ & $0$ & $0$ & $1$ & $1$ & $2$ & $3$ & $6$ \\
%\hline
%\end{tabular}

\noindent We also note that
\begin{align*}
\sum\limits_{n = 0}^{\infty}d_{o}(n)q^{n}  & = \frac{(q^{2};q^{4})_{\infty}}{(q;q)_{\infty}}\sum\limits_{n = 1}^{\infty} \frac{q^{4n - 2}}{1 - q^{4n - 2}}  \\
                                                                   & = \frac{1}{(-q^{2}; q^{2})_{\infty}(q;q)_{\infty}}\sum\limits_{n = 1}^{\infty} \frac{q^{4n - 2}}{1 - q^{4n - 2}} \\
                                                                   & \equiv  \frac{1}{(q^{2}; q^{2})_{\infty}(q;q)_{\infty}}\left(\sum\limits_{n = 1}^{\infty} \frac{q^{2n}}{1 - q^{2n}} - \sum\limits_{n = 1}^{\infty} \frac{q^{4n}}{1 - q^{4n}}\right) \pmod{2} \\
                                                                   & \equiv  \frac{1}{ (q;q)_{\infty}^{3}}\left(\sum\limits_{n = 1}^{\infty} \frac{q^{2n}}{1 - q^{2n}} - \sum\limits_{n = 1}^{\infty} \frac{q^{4n}}{1 - q^{4n}}\right) \pmod{2}
\end{align*}
 so that
\begin{equation}\label{sigma}
\sum\limits_{n = 0}^{\infty}d_{o}(n)q^{n} \sum_{n = 0}^{\infty}q^{n(n+1)/2}  \equiv  \sum\limits_{n = 1}^{\infty} \frac{q^{2n}}{1 - q^{2n}} - \sum\limits_{n = 1}^{\infty} \frac{q^{4n}}{1 - q^{4n}} \pmod{2}
\end{equation} 
where we have used \eqref{three} for $(q;q)_{\infty}^{3}$.\\
\noindent But the right-hand side of \eqref{sigma}  is  
\begin{align*}
 & \sum_{n = 1}^{\infty}\sigma_{0}(n)q^{2n}  - \sum_{n = 1}^{\infty}\sigma_{0}(n)q^{4n}  \\
 & =  \sum_{n > 1, n \equiv 0 \pmod{4}}\sigma_{0}(n/2)q^{n} +   \sum_{n > 1, n \equiv 2 \pmod{4}}^{\infty}\sigma_{0}(n/2)q^{n} - \sum_{n > 1, n \equiv 0 \pmod{4}}^{\infty}\sigma_{0}(n/4)q^{n} \\
& =   \sum_{n > 1, n \equiv 0 \pmod{4}}(\sigma_{0}(n/2) -  \sigma_{0}(n/4))q^{n}  + \sum_{n > 1, n \equiv 2 \pmod{4}}^{\infty}\sigma_{0}(n/2)q^{n}
\end{align*}
If $n \equiv 0 \pmod{4}$, we can write $n = 2^{m}b$ for some positive  integers $m \geq 2$ and $b$ odd. Then
\begin{align*}
\sigma_{0}(n/2) -  \sigma_{0}(n/4)  & = \sigma_{0}(2^{m}b/2) -  \sigma_{0}(2^{m}b/4) \\
 & = \sigma_{0}(2^{m-1}b) -  \sigma_{0}(2^{m-2}b) \\
& = \sigma_{0}(2^{m-1})\sigma_{0}(b) -  \sigma_{0}(2^{m-2})\sigma_{0}(b) \,\,\,\,\,(\text{$\sigma_{0}$ is multiplicative})\\
 & = \sigma_{0}(b)(\sigma_{0}(2^{m-1}) -  \sigma_{0}(2^{m-2})) \\
 & =  \sigma_{0}(b)(m-(m-1)) \\
& = \sigma_{0}(b)
\end{align*}
On the other hand, if $n\equiv 2 \pmod{4}$, we have $n = 2b$ for some odd positive integer $b$. Thus,
 $$ \sigma_{0}(n/2) = \sigma_{0}(b).$$
Denoting by $v_{2}(n)$, the 2-adic valuation of $n$, we obtain the following result.
\begin{corollary} For $n> 0$, we have
$$\sum_{j = 0}^{\lfloor\frac{\sqrt{8n + 1} - 1}{2} \rfloor} d_{o}(n - j(j + 1)/2) 
\equiv
\begin{cases}
0 \pmod{2},&   \text{if $n \equiv 1 \pmod{2}$;}\\\\
\sigma_{0}\left( n/2^{v_{2}(n)} \right) \pmod{2}, & \text{if $n\equiv 0 \pmod 2$}.
\end{cases}
$$
\end{corollary}
\section{Generalizations}
The partition function $d_{o}(n)$ can be generalized as follows:
Let  $k, p \geq 2$  and  $0\leq r < p$ be integers and define $d_{p}(n,k,r)$ as the number of partitions of $n$ in which exactly one  part  $\equiv r \pmod{p}$ appears at least $k$ times. In this definition, parts not congruent to $r \pmod{p}$ appear with unrestricted multiplicity. Similary, let  $f_{p}(n,k,r)$ denote the number of partitions of $n$ in which parts $\equiv kr  \pmod{pk}$ form a singleton set.\\\\
If $r \neq 0$, we have
\begin{align*}
\sum_{n = 0}^{\infty}d_{p}(n,k,r)q^{n} & = \frac{(q^{r};q^{p})_{\infty}}{(q;q)_{\infty}}\sum\limits_{n = 0}^{\infty} \frac{q^{k(pn + r)}}{1 - q^{pn + r}}\prod\limits_{j \neq n, j = 0}^{\infty}(1 + q^{pj + r} + q^{2(pj + r)} + \cdots + q^{(k - 1)(pj + r)} ) \end{align*}
\begin{align*}
                                                          & = \frac{(q^{r};q^{p})_{\infty}}{(q;q)_{\infty}}\sum\limits_{n = 0}^{\infty} \frac{q^{k(pn + r)}}{(1 - q^{pn +  r})\sum\limits_{i = 0}^{k - 1}q^{(pn + r)i}}\prod\limits_{j = 0}^{\infty}\frac{1 - q^{k(pj + r)}}{1 - q^{pj + r}} \\
                                                         & =\frac{(q^{r};q^{p})_{\infty} (q^{kr} ; q^{kp})_{\infty}}{(q^{r};q^{p})_{\infty}(q;q)_{\infty}}  \sum_{n = 0}^{\infty}\frac{q^{pkn + kr}}{1 - q^{pkn + kr}} \\
                                                         & = \frac{(q^{k r} ; q^{kp})_{\infty}}{(q;q)_{\infty}}\sum_{n = 0}^{\infty}\frac{q^{pkn + kr}}{1 - q^{pkn + kr}} \\
                                                         & = \sum_{n = 0}^{\infty}f_{p}(n,k,r)q^{n}.
\end{align*}
If  $r = 0$, then
$$ \sum_{n = 0}^{\infty}d_{p}(n,k,r)q^{n}  = \frac{(q^{r};q^{p})_{\infty}}{(q;q)_{\infty}}\sum\limits_{n = 1}^{\infty} \frac{q^{k(pn + r)}}{1 - q^{pn + r}}\prod\limits_{j \neq n, j = 1}^{\infty}(1 + q^{pj + r} + q^{2(pj + r)} + \cdots + q^{(k - 1)(pj + r)} ) $$
which simplifies to
$$ \frac{(q^{k r} ; q^{kp})_{\infty}}{(q;q)_{\infty}}\sum_{n = 1}^{\infty}\frac{q^{pkn + kr}}{1 - q^{pkn + kr}} =                                                  \sum_{n = 0}^{\infty}f_{p}(n,k,r)q^{n}.$$

Thus, we have:
\begin{theorem}\label{genr}
Let $p\geq 2$.  For an integer $n \geq 1$,
$$ f_{p}(n,k,r) = d_{p}(n,k,r). $$
\end{theorem}
\noindent A bijection (which extends the one for $d_{e}(n)/f_{0}(n)$  and $d_{o}(n)/f_{2}(n)$), is given as follows: \\\\
\noindent Let $\lambda = \left(\lambda_{1}^{m_{1}}, \lambda_{2}^{m_{2}}, \hdots, \lambda_{l}^{m_{l}}\right)$ be a partition enumerated by $f_{p}(n,k,r)$. Then 
$$
\lambda_{i}^{m_{i}} 
\mapsto
 \begin{cases}
\left(\dfrac{\lambda_{i}}{k}\right)^{km_{i}}, & \text{if $\lambda_{i} \equiv 0,k,2k, \ldots, (p-1)k \pmod{pk}$}; \\\\
\phi_{k}^{-1}(\lambda_{i}^{m_{i}}), & \text{otherwise}.
\end{cases}
$$
From $\lambda$, note that there is only one part  congruent to $kr \pmod{pk}$ (with unrestricted multiplicity). Under the bijection, this part is mapped to the exactly one part which is congruent to $r\pmod{p}$ repeated at least $k$ times. For other parts  of $\lambda$ that are not congruent to $ kr \pmod{kp}$, there are two cases.\\
\noindent Case 1:  If the part is divisible by $k$, it can be shown that such a part must be congruent to $jk \pmod{pk}$ where $j\in \{0,1,2,\ldots, p - 1\}\setminus \{ r \}$. This part is mapped to a part not congruent to $r \pmod{p}$ with multiplicity at least $k$.  \\
\noindent Case 2: If that part of $\lambda$  is not divisible by $k$, we apply $\phi_{k}^{-1}$.\\\\
\noindent Conversely, let $\mu = \left(\mu_{1}^{s_{1}}, \mu_{2}^{s_{2}}, \hdots, \mu_{l}^{s_{l}} \right)$ be a partition enumerated by $d_{p}(n,k,r)$. Then 
\begin{equation}\label{inverse}
\mu_{i}^{s_{i}} \mapsto  (k\mu_{i})^{\lfloor \frac{s_i}{k}\rfloor } \cup \phi_{k}(\mu_{i}^{ s_i - k\lfloor \frac{s_i}{k}\rfloor })
\end{equation}
represents the inverse mapping.

\begin{example}
Consider $n = 9$, $p = 3$, $k = 4$ and $r = 1$.
\end{example}
\noindent The $d_{3}(9,4,1)$-partitions are:
$$ (5, 1^{4}), (4,1^{5}), (3, 2, 1^{4}), (3, 1^{6}), (2^{2}, 1^{5}), (2, 1^{7}),  (1^{9}).$$ 
To find the image of $(1^{9})$, we perform the transformation: $$\left\lfloor \frac{9}{4} \right\rfloor = 2\,\,\,\text{and}\,\,\, 9 - 4\left\lfloor \frac{9}{4} \right\rfloor = 1.$$  Applying the inverse map in \eqref{inverse} yields $$ 1^{9} \mapsto (4^{2},1).$$ 
\noindent Similarly, we have
\[ (5, 1^{4}) \mapsto (5, 4) \]
\[ (4, 1^{5})  \mapsto (4, 1^{5}) \]
\[ (3, 2, 1^{4}) \mapsto (4, 3, 2) \]
\[ (3, 1^{6}) \mapsto (4, 3, 1^{2}) \]
\[ (2^{2}, 1^{5}) \mapsto (4, 2^{2}, 1) \]
\[ (2, 1^{7}) \mapsto (4, 2, 1^{3}).\]

\subsection{Connection with the work of Fu and Tang}
Recall that $o_{k}(n)$ denotes the number of partitions of $n$ in which the set of parts congruent to 0 modulo $k$ is singleton.  Fu and Tang \cite{fu} showed that  if $d_{k}(n)$ is the number of partitions of $n$ where exactly one part is repeated at least $k$ times, then
\begin{equation}\label{fueq}
o_{k}(n) = d_{k}(n).
\end{equation} 
\noindent Setting $k = 4$ in \eqref{fueq}, $o_{4}(n) = d_{4}(n)$. On the other hand, we have shown that $f_{0}(n) = d_{e}(n)$ in Theorem \ref{var0}. Since $f_{0}(n) = o_{4}(n)$, we have the identity
\begin{equation}\label{ident3}
d_{4}(n) = d_{e}(n).
\end{equation}
Note that \eqref{ident3} is not \lq trivial\rq . As a preparation for a more generalized result, we describe its bijective proof.\\
\noindent Let $\lambda$ be enumerated by $d_{4}(n)$ and $j$ be the part that is repeated at least 4 times.  Then $ \lambda = \bar{\lambda} \cup j^{m}$ where $m$ is the multiplicity of $j$ in $\lambda$ and $\bar{\lambda}$ is the subpartition of $\lambda$ consisting of parts that are repeated at most 3 times.   Write $m$ as
$$ m = 4q + i\,\,\,\,\,\text{where}\,\,\,\, 0\leq  i\leq 3 \,\,\,\text{and}\,\,\,\, q \geq 1.$$
One writes $j^{m} =  j^{4q} \cup j^{i}$ and converts $j^{4q}$ into $(2j)^{2q}$. \\
Apply Glaisher map $\phi_{4}$ on $\bar{\lambda}\cup j^{i}$ so that the image $\phi_{4}(\bar{\lambda}\cup j^{i})$ is a partition into parts not divisible by 4. \\
Decompose the partition as  $$\phi_{4}(\bar{\lambda}\cup j^{i}) = \lambda^{\prime} \cup \lambda^{\prime\prime}$$ where  $\lambda^{\prime}$  is the subpartition consisting of odd parts and 
$\lambda^{\prime\prime}$ is the subpartition consisting of parts $\equiv 2 \pmod{4}$. Divide each part of $\lambda^{\prime\prime}$ by 2 and apply $\phi_{2}^{-1}$ on the resulting partition.  What follows is a partition into distinct parts. Then multiply every part of this partition by 2 and call the resulting partition $\beta$. \\
Note that
$$(2j)^{2q}\cup \beta \cup \lambda^{\prime}$$ is a partitition enumerated by $d_{e}(n)$.\\\\
The inverse will be demonstrated in the general map. We require the following lemma in the subsequent work.
\begin{lemma}\label{lem}
Let $p$, $k$ and $x$ be positive integers such that $pk \geq 2$. Suppose that $x \not\equiv 0 \pmod{pk}$ and  $x \equiv 0 \pmod{p}$. Then
$\frac{x}{p} \not\equiv 0 \pmod{k}$.
\end{lemma}
%\begin{proof}
%Since  $x \equiv 0 \pmod{p}$, we have 
%$$ x = ip + pkt_{i}\,\,\,\,\text{with}\,\,\, i = 1,2,3, \ldots, k - 1, \,\,\,t_{i} \in \mathbb{Z}. $$
%Thus
%$\frac{x}{p} = i + kt_{i}$ which implies 
%$$\frac{x}{p} \equiv 1,2,3, \ldots, k - 1 \pmod{k}$$
%and the result follows.
%\end{proof}
Now, setting $r = 0$ in Theorem \ref{genr}, observe that $d_{p}(n,k,0)$ is the number of partitions of $n$  in which exactly one part $\equiv 0 \pmod{p}$ appears at least $k$ times. By the above result of Fu and Tang, the following generalization of  \eqref{ident3} follows:
\begin{theorem}
For $n\geq 0$, 
$$d_{pk}(n) = d_{p}(n,k,0).$$
\end{theorem}
Our interest in this theorem is in its bijective construction, which extends the given mapping for the special case $k = 2$ and $p = 2$.\\
 \noindent Let $\lambda$ be enumerated by $d_{pk}(n)$ and $j$ be the part that is repeated at least $pk$ times.  Then $ \lambda = \bar{\lambda} \cup j^{m}$ where is the multiplicity of $j$ in $\lambda$ and $\bar{\lambda}$ is the subpartition of $\lambda$ consisting of parts that are repeated at most $pk - 1$ times.   Write $m$ as
$$ m = pkq + i\,\,\,\,\,\text{where}\,\,\,\, 0\leq  i\leq pk - 1 \,\,\,\text{and}\,\,\,\, q \geq 1.$$
One rewrites $ j^{m}$  as $j^{m} =  j^{pkq} \cup j^{i}$ and converts $j^{pkq}$ into $(pj)^{kq}$. \\
Apply Glaisher map $\phi_{pk}$ on $\bar{\lambda}\cup j^{i}$ so that the image $\phi_{pk}(\bar{\lambda}\cup j^{i})$ is a partition into parts not divisible by  $pk$. \\
Decompose  this partition as  $$\phi_{pk}(\bar{\lambda}\cup j^{i}) = \lambda^{\prime} \cup \lambda^{\prime\prime}$$ where  $\lambda^{\prime}$  is the subpartition consisting of parts $\not\equiv 0 \pmod{p}$ and 
$\lambda^{\prime\prime}$ is the subpartition consisting of parts $\equiv 0\pmod{p}$.  Divide each part of $\lambda^{\prime\prime}$ by $p$ and note that, by Lemma \ref{lem}, the resulting parts are not divisible by $k$. Apply $\phi_{k}^{-1}$ on the resulting partition.  What follows is a partition into parts that appear at most $k - 1$ times . Then multiply every part of this partition by $p$ and call the resulting partition $\beta$. \\
Note that
$$(pj)^{kq}\cup \beta \cup \lambda^{\prime}$$ is a partitition enumerated by $d_{p}(n,k,0)$.
\begin{example}
Consider $n = 433$, $p = 3$ and $k = 4$. 
\end{example}
\noindent Consider the $d_{12}(433)$-partition:
 $$ (13^{10}, 10^{5}, 7^{30}, 6^{2}, 4^{5}, 1^{11})$$
which is decomposed as:
 $$ (13^{10},10^{5},6^{2},4^{5},1^{11}) \cup (7^{30}) $$
where $j=7$ and $m=30=12\cdot 2 + 6$. Thus $7^{30} = (7^{24}) \cup (7^{6})$ and so
 $$ (7^{24}) \mapsto  (21^{8}). $$
Applying the Glaisher's map to the remaining parts yields
 $$ \phi_{12}((13^{10},10^{5},7^{6},6^{2},4^{5},1^{11}) = (13^{10},10^{5},7^{6},6^{2},4^{5},1^{11}). $$
Note that $6 \equiv 0 \pmod{3}$ and thus
 $$ \beta = 3 \times \phi_{4}^{-1}\left(\frac{6}{3},\frac{6}{3}\right) = (6^{2}). $$ 
Taking the union of all the image parts gives
 $$(21^{8},13^{10},10^{5},6^{2},4^{5},1^{11})$$
which is a $d_{3}(433,4,0)$-partition.  \\\\
\noindent The inverse is described as follows: \\
Let $\mu$ be enumerated by  $d_{p}(n,k,0)$. Decompose the partition  as  $$\mu = s^{m} \cup \beta \cup \lambda^{\prime}$$ where $s$ is that one part  $\equiv 0 \pmod p$ whose multiplicity $m$ is at least $k$ times, $\beta$ is the subpartition consisting of parts $ 
\equiv 0 \pmod{p} $ whose multiplicities are $\leq k - 1$ and $\lambda^{\prime}$ is the subpartition consisting of parts $\not\equiv 0 \pmod{p}$.
Write $m$ as
$$ m = kt + f\,\,\,\,\,\text{where}\,\,\,\, 0\leq  f\leq k - 1 \,\,\,\text{and}\,\,\,\, t \geq 1.$$
One rewrites $ s^{m}$  as $s^{m} =  s^{kt} \cup s^{f}$ and converts $s^{kt}$ into $\left(\frac{s}{p}\right)^{pkt}$.  \\
\noindent Divide each part of $\beta \cup s^{f}$  by $p$ and then apply $\phi_{k}$ on the resulting partition. Then multiply every part of the obtained partition by $p$ and denote the resulting partition  by $\mu^{\prime}$. Note  that $\mu^{\prime}$ is a partition whose parts are not divisible by $pk$  but divisible by $p$. Thus the parts  of $\lambda^{\prime} \cup \mu^{\prime}$  are not divisible by $pk$. Now compute 
 $$ \mu^{\prime\prime} = \phi_{pk}^{-1}(\lambda^{\prime} \cup \mu^{\prime} )$$
so that  $\mu^{\prime\prime}$ is a partition with parts appearing at most  $pk- 1$ times.
Clearly,
$$ \left(\frac{s}{p}\right)^{pkt}\cup \mu^{\prime\prime}$$ is a partition enumerated by $d_{pk}(n)$. \\\\

\noindent Reversing the previous example, recall the $d_{3}(433,4,0)$-partition:  
$$ (21^{8},13^{10},10^{5},6^{2},4^{5},1^{11}).$$
We decompose it follows:
$$ (21^{8}) \cup (6^{2}) \cup (13^{10},10^{5},4^{5},1^{11}). $$
Thus,
 $$ (21^{8}) \mapsto (7^{24}),$$
 $$ 3 \times \phi_{4}\left(\frac{6}{3},\frac{6}{3}\right) = (6^{2}). $$ 
Applying the Glaisher's map to $(13^{10},10^{5},4^{5},1^{11}) \cup (6^{2})$ gives 
$$ \phi_{12}^{-1}(13^{10},10^{5},6^{2},4^{5},1^{11}) = (13^{10},10^{5},7^{6},6^{2},4^{5},1^{11}). $$ 
Now, taking the union of all the image parts yields
 $$ (13^{10},10^{5},7^{30}, 6^{2},4^{5},1^{11}) $$ 
which is a $d_{12}(433)$-partition.

\subsection{A slightly different partition function}
For $\alpha \geq k$, let $g(n,\alpha, k,p)$ denote  the number of partitions of $n$ in which only one part appears at least  $\alpha$ times  and its multiplicity is congruent to $\alpha + j$ modulo $p$ where  $0\leq j \leq k - 1$ and all other parts appear at most $k - 1$ times. Denote by $g_{o}(n,\alpha, k,p)$ (resp. $g_{e}(n,\alpha, k,p)$ the number of $g(n,\alpha, k,p)$-partitions in which the part repeated at least $\alpha$ times is odd (resp. even). Then we have:
\begin{theorem}
\begin{equation}\label{big}
\sum\limits_{n = 0}^{\infty}g_{o}(n,\alpha, k,p) - g_{e}(n,\alpha, k,p)q^{n} = \frac{(q^{k};q^{k})_{\infty}}{(q;q)_{\infty}} \sum_{n  =0 }^{\infty}\frac{q^{pn + \alpha}}{1 + q^{pn + \alpha}}.
\end{equation}
\end{theorem}
\begin{proof}
If $z$ tracks the part repeated at least $\alpha$ times, observe that
\begin{align*}
 & \sum_{m = 0}^{\infty}\sum_{n = 0}^{\infty}g(n,\alpha, k,p) q^{n}z^{m}\\
  & = \sum_{n = 1}^{\infty}\left( z^{n}q^{\alpha n} + z^{n}q^{(\alpha + p)n} + z^{n}q^{(\alpha + 2p)n} +z^{n}q^{(\alpha + 3p)n} +  z^{n}q^{(\alpha + 4p)n} + \right.\\
& \left.\,\,\,\,\,\,\,\,\,\,\,\,\,\,\,  \cdots   + z^{n}q^{(\alpha + 1 + p)n} + z^{n}q^{(\alpha + 1 + 2p)n}  + z^{n}q^{(\alpha +1 + 3p)n} + z^{n}q^{(\alpha +1 + 4p)n} +  \right. \\
& \left. \,\,\,\,\,\,\,\,\,\,\,\,\,\,\,  \cdots  + z^{n}q^{(\alpha + 2 + p)n} + z^{n}q^{(\alpha + 2 + 2p)n}  + z^{n}q^{(\alpha +2 + 3p)n} + z^{n}q^{(\alpha +2 + 4p)n} +  \right. \\
& \left. \,\,\,\,\,\,\,\,\,\,\,\,\,\,\,  \cdots  + z^{n}q^{(\alpha + 3 + p)n} + z^{n}q^{(\alpha + 3 + 2p)n}  + z^{n}q^{(\alpha +3 + 3p)n} + z^{n}q^{(\alpha +3 + 4p)n} +  \right. \\
& \left.  \,\,\,\,\,\,\,\,\,\,\,\,\,\,\,\,\,\,\,\,\,\,\,\,\,\,\, \vdots \right. \\
& \left. \,\,\,\,\,\,\,\,\,\,\,\,\,\,\,  \cdots  + z^{n}q^{(\alpha + k - 2 + p)n} + z^{n}q^{(\alpha + k - 2 + 2p)n}  + z^{n}q^{(\alpha + k- 2 + 3p)n} + z^{n}q^{(\alpha + k- 2 + 4p)n} +  \right. \\
& \left. \,\,\,\,\,\,\,\,\,\,\,\,\,\,\,  \cdots  + z^{n}q^{(\alpha + k - 1 + p)n} + z^{n}q^{(\alpha + k - 1 + 2p)n}  + z^{n}q^{(\alpha +k - 1 + 3p)n} + z^{n}q^{(\alpha + k - 1 + 4p)n} +  \right. \\
& \left. \,\,\,\,\,\,\,\,\,\,\,\,\,\,\,  \cdots  \right) \prod_{j = 1, j \neq n}^{\infty}(1 + q^{j} + q^{2j} + q^{3j} + \cdots + q^{(k -1)j}) \\
& = \sum_{n = 1}^{\infty}z^{n}\sum_{i = 0}^{k - 1}\sum_{m = 0}^{\infty}q^{(\alpha + i + mp)n}\prod_{j = 1, j \neq n}^{\infty}(1 + q^{j} + q^{2j} + q^{3j} + \cdots + q^{(k -1)j}) \\
& = \sum_{n = 1}^{\infty}z^{n}q^{\alpha n}\sum_{i = 0}^{k - 1}q^{in}\sum_{m = 0}^{\infty}q^{mpn}\prod_{j = 1, j \neq n}^{\infty}(1 + q^{j} + q^{2j} + q^{3j} + \cdots + q^{(k -1)j}) \\
& = \sum_{n = 1}^{\infty}\frac{z^{n}q^{\alpha n}}{1 - q^{pn}}\sum_{i = 0}^{k - 1}q^{in}\prod_{j = 1, j \neq n}^{\infty}(1 + q^{j} + q^{2j} + q^{3j} + \cdots + q^{(k -1)j}).  \\
& = \sum_{n = 1}^{\infty}\frac{z^{n}q^{\alpha n}}{1 - q^{pn}}\prod_{j = 1}^{\infty}(1 + q^{j} + q^{2j} + q^{3j} + \cdots + q^{(k -1)j}).  \\
& = \sum_{n = 1}^{\infty}\frac{z^{n}q^{\alpha n}}{1 - q^{pn}}\prod_{j = 1}^{\infty}\frac{1 - q^{kj}}{1 - q^{j}} \\
& = \frac{(q^{k};q^{k})_{\infty}}{(q;q)_{\infty}}\sum_{n = 1}^{\infty}\frac{z^{n}q^{\alpha n}}{1 - q^{pn}}
\end{align*} 
so that 
\begin{align*}
\sum\limits_{n = 0}^{\infty}(g_{o}(n,\alpha, k,p) - g_{e}(n,\alpha, k,p))q^{n} & = - \frac{(q^{k};q^{k})_{\infty}}{(q;q)_{\infty}}\sum_{n = 1}^{\infty}\frac{(-1)^{n}q^{\alpha n}}{1 - q^{pn}}\\
                                                                                                                  & = - \frac{(q^{k};q^{k})_{\infty}}{(q;q)_{\infty}}\sum_{n = 1}^{\infty}(-1)^{n}q^{\alpha n}\sum_{m = 0}^{\infty}q^{pnm} \\
                                                                                                                  & = - \frac{(q^{k};q^{k})_{\infty}}{(q;q)_{\infty}}\sum_{m = 0}^{\infty}\sum_{n = 1}^{\infty}(-q^{\alpha + pm })^{n} \end{align*}
\begin{align*}
                                                                                                           & = - \frac{(q^{k};q^{k})_{\infty}}{(q;q)_{\infty}}\sum_{m = 0}^{\infty} \frac{-q^{\alpha + pm }}{1 + q^{\alpha + pm }} \\
                                                                                                                  & =  \frac{(q^{k};q^{k})_{\infty}}{(q;q)_{\infty}}\sum_{m = 0}^{\infty} \frac{q^{\alpha + pm }}{1 + q^{\alpha + pm }}.
\end{align*}
\end{proof}
We have the following relationship between $h(n,p)$ and $g(n,p,p,p)$.
\begin{proposition}\label{gh}
For all $n\geq 0$, 
$$ g_{o}(n,p, p,p)  =  h_{0}(n,p) \,\,\,\text{and}\,\,\,  g_{e}(n,p, p,p) = h_{p}(n,p).$$
\end{proposition}
\begin{proof}
From the fact that  $$ \sum_{m = 0}^{\infty}\sum_{n = 0}^{\infty}g(n,p, p,p) q^{n}z^{m} = \frac{(q^{p};q^{p})_{\infty}}{(q;q)_{\infty}}\sum_{n = 1}^{\infty}\frac{z^{n}q^{pn}}{1 - q^{pn}},$$
recall that $z$ tracks the part which is repeated at least $p$ times. Then
 \begin{align*}
\sum_{n = 0}^{\infty}g_{o}(n,p, p,p)q^{n} 
                                                                             & = \frac{(q^{p};q^{p})_{\infty}}{(q;q)_{\infty}}\sum_{n = 1, n \equiv 1 \pmod{2}}^{\infty} \frac{q^{pn }}{1 - q^{pn }}\\
                                                                             & =  \frac{(q^{p};q^{p})_{\infty}}{(q;q)_{\infty}}\sum_{n = 1}^{\infty} \frac{q^{2pn - p }}{1 - q^{2pn - p }}\\
                                                                             & = \sum_{n = 0}^{\infty}h_{0}(n,p) q^{n}
\end{align*}
and on the same note,
\begin{align*}
\sum_{n = 0}^{\infty}g_{e}(n,p, p,p)q^{n} 
                                                                              & = \frac{(q^{p};q^{p})_{\infty}}{(q;q)_{\infty}}\sum_{n = 1, n \equiv 0 \pmod{2}}^{\infty} \frac{q^{pn }}{1 - q^{pn }}\\
                                                                             & =  \frac{(q^{p};q^{p})_{\infty}}{(q;q)_{\infty}}\sum_{n = 1}^{\infty} \frac{q^{2pn }}{1 - q^{2pn  }}\\
                                                                             & = \sum_{n = 0}^{\infty}h_{p}(n,p) q^{n}.
\end{align*}
\end{proof}
%\section{Declarations}
%\subsection{Competing Interests and/or Funding}
%On behalf of all authors, the corresponding author states that there is no conflict of interest. No funding was received for conducting this study. The authors have no financial or proprietary interests in any material discussed in this article. 
%\subsection{ Data Availability Statement}
%Authors can confirm that this manuscript has no associated data.

\end{document}